\documentclass[12pt,reqno]{amsart}

\usepackage{times,a4wide,mathrsfs,amssymb,amsmath,amsthm,enumerate,xypic,dsfont}
\usepackage{stmaryrd}

\usepackage{hyperref}
\usepackage[usenames,dvipsnames]{xcolor}
\usepackage{tikz-cd}

\hypersetup{colorlinks=true,citecolor=NavyBlue,linkcolor=Brown,urlcolor=Orange}

\newcommand{\FFF}{\mathbb{F}}
\newcommand{\C}{\mathbb{C}}
\newcommand{\ZZ}{\mathbb{Z}}

\newcommand{\QQ}{\mathbb{Q}}

\newcommand{\PP}{\mathbb{P}}

\newcommand{\oo}{\mathfrak o}

\newcommand{\II}{\mathcal I}

\newcommand{\wt}{\widetilde}
\newcommand{\rom}{\romannumeral}

\renewcommand{\v}{\mathbf{v}}
\newcommand{\w}{\mathbf{w}}

\DeclareMathOperator{\CH}{CH}

\renewcommand{\1}{\mathds{1}}
\newcommand{\h}{\mathfrak{h}}

\DeclareMathOperator{\ide}{id}

\DeclareMathOperator{\ima}{Im}

\DeclareMathOperator{\Gr}{Gr}

\DeclareMathOperator{\Ku}{Ku}

\DeclareMathOperator{\HH}{H}
\newcommand{\cart}{\ar@{}[dr]|\square} 

\DeclareMathOperator{\CHM}{CHM}

\DeclareMathOperator{\Db}{D^b}

\DeclareMathOperator{\mar}{mar}
\DeclareMathOperator{\lab}{lab}
\DeclareMathOperator{\GDCH}{GDCH}

\newtheorem{theorem}{Theorem}[section]

\newtheorem{lemma}[theorem]{Lemma}

\newtheorem{corollary}[theorem]{Corollary}
\newtheorem{proposition}[theorem]{Proposition}
\newtheorem{conjecture}[theorem]{Conjecture}

\theoremstyle{definition}
\newtheorem{remark}[theorem]{Remark}
\newtheorem{definition}[theorem]{Definition}
\newtheorem{question}[theorem]{Question}
\newtheorem{example}[theorem]{Example}

\newif\ifHideFoot
\HideFoottrue  
\HideFootfalse  

\ifHideFoot

\newcommand{\Lie}[1]{}
\newcommand{\Robert}[1]{}

\else 

\newcommand{\marg}[1]{\normalsize{{
			\color{red}\footnote{{\color{blue}#1}}}{\marginpar[\vskip
			-.25cm{\color{red}\hfill$\Rightarrow$\tiny\thefootnote}]{\vskip
				-.2cm{\color{red}$\Leftarrow$\tiny\thefootnote}}}}}
\newcommand{\Lie}[1]{\marg{(Lie) #1}}
\newcommand{\Robert}[1]{\marg{(Robert) #1}}

\fi

\title[Special cubic 4-folds, K3 surfaces and Franchetta property]{Special cubic fourfolds, K3 surfaces and the Franchetta property}

\author[Lie Fu]{Lie Fu}
\address{Institute for Mathematics, Astrophysics and Particle Physics (IMAPP), Radboud University, PO Box 9010, 6500 GL, Nijmegen, Netherlands.}
\email{lie.fu@math.ru.nl}

\author[Robert Laterveer]
{Robert Laterveer}

\address{Institut de Recherche Math\'ematique Avanc\'ee,
	CNRS -- Universit\'e 
	de Strasbourg,\
	7 Rue Ren\'e Des\-car\-tes, 67084 Strasbourg CEDEX,
	FRANCE.}
\email{robert.laterveer@math.unistra.fr}

\thanks{L.F. and R.L. are supported by the Agence Nationale de la Recherche (ANR) under project number ANR-20-CE40-0023. L.F. was also supported by the  ANR project ANR-16-CE40-0011 and the Radboud Excellence Initiative program.}

\begin{document}

\begin{abstract} 
	O'Grady conjectured that the Chow group of 0-cycles of the generic fiber of the universal family over the moduli space of polarized K3 surfaces of genus $g$ is cyclic. This so-called generalized Franchetta conjecture has been solved only for low genera where there is a Mukai model (precisely, when $g\leq 10$ and $g=12, 13, 16, 18, 20$), by the work of Pavic--Shen--Yin. 
	In this paper, as a non-commutative analogue, we study the Franchetta property for families of special cubic fourfolds (in the sense of Hassett), and relate it
	to O'Grady's conjecture for K3 surfaces. 
	Most notably, by using special cubic fourfolds of discriminant 26, we prove O'Grady's generalized Franchetta conjecture for $g=14$, providing the first evidence beyond Mukai models.
 \end{abstract}

\maketitle

\tableofcontents

\section{Introduction}

For an integer $g\geq 2$, let $\mathcal{M}_g$ be the moduli stack of genus $g$ curves and $\pi\colon \mathcal{C}\to \mathcal{M}_g$ the universal curve. 
Franchetta conjectured in \cite{Franchetta54} that the Picard group of the generic fiber of $\pi$ is free cyclic and generated by the relative canonical bundle  $\omega_\pi$.
The conjecture can be equivalently formulated as follows: for any line bundle $L$ on $\mathcal{C}$, the restriction of $L$ to a fiber $C_t:=\pi^{-1}(t)$, for any $t\in \mathcal{M}_g$, is a power of the canonical bundle:
\[L|_{C_t}\simeq \omega_{C_t}^{\otimes m}, \text{ for some } m\in \ZZ.\] Franchetta's conjecture was proved  by Harer \cite{Harer83} (see also \cite{ArabarelloCornalba87} and \cite{Mestrano}).

\subsection{Franchetta for K3 surfaces}
In the end of \cite{OGradyK3}, O'Grady proposed an analogue of Franchetta's conjecture for K3 surfaces. In order to state his conjecture, let us first recall the following seminal result of Beauville and Voisin \cite{BVK3}. Let $\CH^\ast(-)$ denote the Chow ring.
\begin{theorem}[Beauville--Voisin]
	Let $S$ be a projective K3 surface. There exists a canonical 0-cycle $\oo_S\in \CH^2(S)$, defined as the class of any point lying on some rational curve in $S$, satisfying the following properties:
	\begin{enumerate}[(i)]
		\item $ \ima\left( \CH^1(S)\otimes \CH^1(S)\xrightarrow{\cdot} \CH^2(S)\right) \subset \ZZ \oo_S.$
		\item $c_2(T_S)=24\oo_S$ in $\CH^2(S)$.
	\end{enumerate}
\end{theorem}

We call the canonical 0-cycle $\oo_S$ the \textit{Beauville--Voisin class} of the K3 surface $S$. The existence of such a canonical class is remarkable, as Mumford proved in \cite{Mumford68} that $\CH^2(S)$ is infinite dimensional, in the sense that it cannot be parameterized by a scheme of finite type. The insight of O'Grady is that to generalize Franchetta's conjecture,  the Beauville--Voisin class for a K3 surface should play the role of the canonical class for a curve.

Now let us state O'Grady's conjecture in \cite[p. 717]{OGradyK3} precisely. Throughout the paper, for an integer $g\geq 2$, we denote by $\mathcal{F}_g$ the moduli stack of primitively polarized K3 surfaces of genus $g$, that is, a pair
$(S, H)$ of a K3 surface $S$ and a primitive ample line bundle $H$  on it with degree $(H^2)=2g-2$. Let $\pi\colon \mathcal{S}\to \mathcal{F}_g$ be the universal family ($\mathcal{S}$ is sometimes denoted by $\mathcal{F}_{g,1}$ in the literature). For any closed point $b\in \mathcal{F}_g$, we denote by $S_b$ the fiber of $\pi$ over $b$. Rational Chow groups of algebraic stacks are defined in \cite{Vistoli97}.

\begin{conjecture}[O'Grady]
	\label{conj:GFC}
For any $b\in \mathcal{F}_g$, the Gysin restriction of any cycle $z\in \CH^2(\mathcal{S})_\QQ$ to the fiber $S_b$ is a multiple of the Beauville--Voisin class, i.e.,
 \[\ima\left(\CH^2(\mathcal{S})_\QQ\to \CH^2(S_b)_\QQ\right)=\QQ\oo_{S_b}.\]
\end{conjecture}  
We will refer to this conjecture as the \textbf{generalized Franchetta conjecture}. Note that by the standard argument of ``spreading out'' (see for example \cite[Section 1.1.2]{VoisinChowBook}), it is equivalent to requiring the same property only for a \textit{very general} point $b$ in $\mathcal{F}_g$. 

Conjecture \ref{conj:GFC} is largely open at present. Let us first mention some closely related results:
\begin{itemize}
	\item Bergeron and Li  \cite[Theorem 1.2.1]{BergeronLi} established a cohomological version of the conjecture: for any $z\in \CH^2(\mathcal{S})_\QQ$, if it is cohomologically  trivial on each fiber of $\pi$, then its cohomology class $[z]$ vanishes on the preimage of a Zariski open subset of $\mathcal{F}_g$. 
	\item Beauville recently proved in \cite{Beauville-Franchetta} that for any $g$, there exists a hypersurface in $\mathcal{F}_g$ such that the restricted universal family satisfies the Franchetta property, in the sense of Definition \ref{def:Franchetta} below.
	\item In a series of joint work with Vial \cite{FLV-Franchetta1, FLV-MCK, FLV-Franchetta2}, we formulated and investigated the natural extension of Conjecture \ref{conj:GFC} for higher-dimensional hyper-K\"ahler varieties, which is proved most notably in the cases of Beauville--Donagi fourfolds \cite{BeauvilleDonagi} and Lehn--Lehn--Sorger--van Straten eightfolds \cite{LLSvS} associated with the universal family of cubic fourfolds. 
\end{itemize}  
As for Conjecture \ref{conj:GFC} itself, the only known result so far is the following:

\begin{theorem}[{Pavic--Shen--Yin \cite{PavicShenYin-FranchettaK3}}] 
	\label{thm:PavicShenYin}
	Conjecture \ref{conj:GFC} is true for $2\le g\le 10$ and for $g\in\{12, 13, 16,18,20 \}$. 
\end{theorem}

The values of $g$ appearing in the statement are exactly the ones where a so-called \textit{Mukai model} is available, and indeed, Theorem \ref{thm:PavicShenYin} is proven by exploiting the projective geometry of those Mukai models. Here, a Mukai model refers to a description of a general genus $g$ polarized K3 surface as the zero locus of a general section of some globally generated homogeneous vector bundle over a homogeneous variety. Examples are double covers of $\PP^2$ ramified along a sextic curve (for $g=2$), quartic surfaces in $\PP^3$ (for $g=3$), complete intersections of a hyperquadric and a cubic hypersurface in $\PP^4$ (for $g=4$), complete intersections of three hyperquadrics in $\PP^5$ (for $g=5$), complete intersections of three hyperplanes and a hyperquadric with $\Gr(2,5)$ embedded in $\PP^9$ via Pl\"ucker (for $g=6$), and so on. For more details on the geometric constructions, we refer to the original papers of Mukai \cite{Mukai88, Mukai89, Mukai06, Mukai16}, and also to \cite[Section 2]{PavicShenYin-FranchettaK3} for a summary. \\

Our main result is the following, which provides the first instance of the generalized Franchetta conjecture  \ref{conj:GFC} beyond Mukai models:
\begin{theorem}
	\label{thm:main-1}
	Conjecture \ref{conj:GFC} holds for $g=14$.
\end{theorem}

What is probably more interesting than the result is our approach to establishing it. Theorem \ref{thm:main-1} is implied by the combination of Theorem \ref{thm:main-3} and Theorem \ref{thm:main-2} below. Let us now give a brief account.

\subsection{Franchetta for special cubic fourfolds}
Special cubic fourfolds were first introduced and studied by Hassett \cite{Hassett-SpeicialCubic}. These are cubic fourfolds $X$ containing a surface $R$ whose class is not proportional to $h^2$, the square of the hyperplane class. Special cubic fourfolds come in families enumerated by the discriminant $d$ of the sublattice of $H^4(X, \mathbb{Z})$ generated by $R$ and $h^2$. The moduli space of special cubic fourfolds of discriminant $d$ is denoted by  $\mathcal{C}_d$, which is non-empty and irreducible when $d\equiv 0, 2 \pmod 6$. For $d$ satisfying an extra numerical condition $(**)$ (see Section \ref{subsec:AssociatedK3}), a special cubic fourfold $X$ of discriminant $d$ has an associated K3 surface $S$, such that $X$ and $S$ are related Hodge theoretically (\cite{Hassett-SpeicialCubic}), and it turns out there are also  strong relations between their derived categories \cite{Kuz10} \cite{AddingtonThomas} and algebraic cycles (or motives) \cite{Buelles}.
All the above is explained in more detail in Section \ref{sect:SpecialCubic}.

The proof of Theorem \ref{thm:main-1}, which uses  special cubic fourfolds of discriminant 26, can be summarized as follows.
Sending such a cubic fourfold to its (Hodge theoretically) associated K3 surface gives a birational isomorphism between $\mathcal{F}_{14}$ and the moduli space $\mathcal{C}_{26}$. Let $U$ be a common Zariski open subset and denote by $\mathcal{S}$ and $\mathcal{X}$ the universal families of K3 surfaces and of cubic fourfolds respectively. Our proof splits into two parts:\\
\textbf{Step 1.} Produce a relative correspondence over $U$ between $\mathcal{S}$ and $\mathcal{X}$, and show that Conjecture~\ref{conj:GFC} for $\mathcal{S}\to \mathcal{F}_{14}$ is equivalent to the \textit{Franchetta property} (Definition \ref{def:Franchetta}) for $\mathcal{X}\to \mathcal{C}_{26}$.\\
\textbf{Step 2.} Establish the Franchetta property for $\mathcal{X}\to \mathcal{C}_{26}$ by using the concrete geometric characterization of such cubic fourfolds due to Farkas--Verra \cite{FarkasVerra-genus14} as the ones containing certain type of scrolls. \\
The upshot is that although there is no Mukai model for K3 surfaces of genus $14$ at our disposal, we have the following replacement which is almost as good: \textit{a generic K3 surface of genus 14 is a moduli space of Bridgeland-stable objects, with certain isotropic Mukai vector,  in the Kuznetsov component of a cubic fourfold that contains a 3-nodal septic rational scroll.}\\

In this paper, both of the above steps are treated in greater generality. For Step 1, which is accomplished in Section \ref{subsect:LinkCubicK3}, we actually give a strong link between the Franchetta properties for special cubic fourfolds and for the associated K3 surfaces. Let us state here only the non-technical version. See Theorem \ref{thm:GFCK3andCubic} for a stronger form. 

\begin{theorem}
	\label{thm:main-3}
	Let $d$ be an integer satisfying the condition $(**)$ (see Section \ref{subsec:AssociatedK3}). Let $g=\frac{d}{2}+1$. If $d\equiv 2 \pmod 6$, then the Franchetta property (Definition \ref{def:Franchetta}) for the universal family over $\mathcal{C}_{d}$ is equivalent to the Franchetta property for the universal family over $\mathcal{F}_g$.
\end{theorem}

In view of Step 2, we are led to ask the following question. As cubic fourfolds are considered as non-commutative analogues of K3 surfaces \cite{Kuz10}, the following can be seen as a non-commutative version of O'Grady's generalized Franchetta conjecture \ref{conj:GFC}.
\begin{question}
	Let $d>6$ be an integer $\equiv 0 \text{ or } 2 \pmod 6 $. Does the universal family $\mathcal{X} \to \mathcal{C}_{d}$  of special cubic fourfolds of discriminant $d$ satisfy the Franchetta property (Definition \ref{def:Franchetta})? That is, for any $b\in \mathcal{C}_d$,
	\begin{equation}\label{eq:FranchettaCubic}
		\ima\left(\CH^3(\mathcal{X})_\QQ\to \CH^3(X_b)_\QQ\right) \overset{?}{=} \QQ h^3.
	\end{equation}
\end{question}
The left-hand side is often denoted by $\GDCH^3_B(X_b)$ is this paper.\\

We answer this question affirmatively in a few cases:
\begin{theorem}
	\label{thm:main-2}
	The Franchetta property \eqref{eq:FranchettaCubic} holds for the universal family of special cubic fourfolds $\mathcal{C}_{d}$ with discriminant $d= 8, 14, 20, 26, 38$.
\end{theorem}

\begin{remark}[Relation with previous results]
	\label{rmk:KnownResults}
	In Theorem \ref{thm:main-2}, the case $d=8$ provides a new proof of Conjecture \ref{conj:GFC} for $g=2$; the case $d=14$ gives a new proof of Conjecture \ref{conj:GFC} for $g=8$ using Theorem \ref{thm:main-3}; the case $d=20$ has been proven using different methods in our previous joint work with Vial  \cite[Lemma 6.3]{FLV-MCK}; the case $d=26$ is the principal case, yielding Theorem \ref{thm:main-1}; the case $d=38$ is proven using Theorem \ref{thm:main-3} and the $g=20$ case of Conjecture \ref{conj:GFC}, demonstrating the flow of information in the reverse direction. Finally, note that the Franchetta property for the  universal family over the whole moduli space of cubic fourfolds can be easily checked (see however \cite[Theorem 2]{FLV-Franchetta2} for stronger and more interesting results).
\end{remark}

\subsection*{Potential and limits}
Cubic fourfolds are instances of the so-called \textit{varieties of K3 type} (see \cite{FLV-MCK}), which means an even-dimensional smooth projective variety $X$ whose Hodge numbers $h^{p,q}(X)=0$ for all $p\neq q$ except for $h^{m-1, m+1}(X)=h^{m+1, m-1}(X)=1$ where $2m=\dim(X)$. The terminology is justified by the observation that its middle cohomology group $\HH^{2m}(X, \ZZ)$, up to a Tate twist, carries a weight-2 Hodge structure of K3 type. Examples of varieties of K3 type include cubic fourfolds, Gushel--Mukai fourfolds and sixfolds 
\cite{Mukai88, DebarreIlievManivel-GM4, KP16, DebarreKuznetsov-GM1}, Debarre--Voisin 20-folds \cite{DebarreVoisin} etc.; see \cite{BFMT-FanoK3} for a recent updated list.

We expect that our approach will lead to further progress on Conjecture~\ref{conj:GFC}: whenever a concrete geometric description is discovered for a family of varieties of K3 type whose generic member has an associated K3 surface which is generic in the moduli space $\mathcal{F}_g$, our argument gives access to the generalized Franchetta conjecture for this $g$. We view Theorem~\ref{thm:main-1} for $g= 14$, as well as our new proofs for $g=2$ and 8 without using Mukai models (Remark \ref{rmk:KnownResults}), merely as the first examples of this approach.

In the past few years, we witnessed a rapid development on the projective geometry of special cubic fourfolds \cite{Nuer17, Lai, FarkasVerra-genus14, RussoStagliano-Duke, BolognesiRussoStagliano-C14, FarkasVerra-genus22},  Gushel--Mukai fourfolds \cite{DebarreIlievManivel-GM4, HoffStagliano-20}, and Debarre--Voisin 20-folds \cite{BenedettiSong}. These achievements will certainly shed light on the geometry of K3 surfaces in the future. 

An initial motivation to construct Mukai models was to prove the \textit{unirationality} of the moduli spaces $\mathcal{F}_g$ for $g$ taking values as in Theorem \ref{thm:PavicShenYin}. Recent progress in this direction is due to Farkas--Verra \cite{FarkasVerra-genus14} (for $g=14$), Farkas--Verra \cite{FarkasVerra-genus22} (for $g=22$), and Hoff--Staglian\`o \cite{HoffStagliano-20} (a new proof for $g=11$, originally due to Mukai \cite{Mukai-genus11}). Both the arguments in \cite{PavicShenYin-FranchettaK3} and in the present paper require ``parameterizing'' K3 surfaces by a flag variety, which in practice always takes the form of a unirational parameterization. 

However, Gritsenko--Hulek--Sankaran \cite{GritsenkoHulekSankaran-KodK3} showed that $\mathcal{F}_g$ is of non-negative Kodaira dimension, hence not unirational, for $g\geq 41$ and $g\neq 42, 45, 46, 48$; see similar results for moduli spaces of special cubic fourfolds in \cite[Proposition 1.3]{Nuer17} and for moduli spaces of special Gushel--Mukai fourfolds in \cite{Petok}. Therefore, for a high genus in this range, some entirely new idea is needed to study the generalized Franchetta conjecture.

\bigskip
The paper is organized as follows: In Section \ref{sec:Generalities}, we collect some basic facts concerning the Franchetta property. In Section \ref{sect:SpecialCubic}, we first recap the theory of special cubic fourfolds and their associated K3 surfaces, then we establish  the bridge between their Franchetta properties, namely, Theorem \ref{thm:main-3} (or rather its more precise version Theorem \ref{thm:GFCK3andCubic}). In the remaining sections, as their titles indicate, we prove Theorem \ref{thm:main-2} case by case and give applications.

\subsection*{Convention\,:} Throughout the paper, we work  over the field of complex numbers $\C$. All Chow groups and Chow motives are with rational coefficients:
for any (possibly singular) variety $X$ of dimension $d$ we write $\CH_i(X)=\CH^{d-i}(X)$ for the group of $i$-dimensional algebraic cycles with $\QQ$-coefficients modulo rational equivalence.
A \textit{lattice} means a free abelian group of finite rank equipped with a symmetric bilinear pairing.

\subsection*{Acknowledgment:} We would like to thank Emma Brakkee, Michael Hoff, Kuan-Wen Lai, Chunyi Li, Francesco Russo, Paolo Stellari, Xiaolei Zhao for helpful discussions.

\section{Franchetta property and generically defined cycles}
\label{sec:Generalities}

To tackle  with the generalized Franchetta conjecture, we will need to study this property beyond the scope of K3 surfaces:

\begin{definition}[Franchetta property \cite{FLV-Franchetta1, FLV-Franchetta2}]
	\label{def:Franchetta}
	Let $\mathcal{X}\to B$ be a smooth projective morphism between complex varieties (or algebraic stacks). For an integer $i\geq 0$, we say that the family $\mathcal{X}/B$ satisfies the \textit{Franchetta property} for codimension-$i$ cycles, if for any $z\in \CH^i(\mathcal{X})_\QQ$ and any $b\in B$, the Gysin restriction $z|_{X_b}$ is rationally equivalent to zero if and only if its cohomology class $[z|_{X_b}]=0$ in $\HH^{2i}(X_b, \QQ)$. If this holds for all $i$, we simply say that $\mathcal{X}/B$ has the Franchetta property. Again, by spreading out rational equivalence, it is equivalent to requiring the same property only for very general $b\in B$.
\end{definition}

\begin{remark}
	\label{rmk:Subfamily}
	Note that there is no implication in either direction between the Franchetta properties for a family $\mathcal{X}\to B$ and for a subfamily $\mathcal{X}_{B'}\to B'$, where $B'$ is a closed subscheme of $B$ (see \cite[p.1]{FLV-Franchetta2}).  However, if $B'\to B$ is a dominant morphism, the Franchetta property of the base-changed family $\mathcal{X}_{B'}\to B'$ implies the Franchetta property for $\mathcal{X}\to B$ (see \cite[Remark 2.6]{FLV-Franchetta1}); in particular, on can freely replace $B$ by a non-empty Zariski open subset. 
\end{remark}

To study the generalized Franchetta conjecture \ref{conj:GFC}, or more generally the Franchetta property (Definition \ref{def:Franchetta}), it is convenient to introduce the following notion.

\begin{definition}[Generically defined cycles]
	let $\pi\colon \mathcal{X}\to B$ be a smooth projective morphism between complex varieties (or algebraic stacks). Let  $X$ be a fiber of $\pi$ over a closed point. We define the group of {\em generically defined cycles} on $X$ as the following graded subgroup of $\CH^*(X)$:
	\[ \GDCH^*_B(X):=\ima\left( \CH^*(\mathcal{X})\to \CH^*(X)\right),\]
	where the morphism is the Gysin restriction map. 
\end{definition}
Using this notation, the Franchetta property (Definition \ref{def:Franchetta}) for $\mathcal{X}/B$ is equivalent to the injectivity of the cycle class map:
\[\GDCH^*_B(X)\to \HH^*(X, \QQ),\]
for all (or equivalently, for very general) fibers $X$.

\medskip
In \cite{PavicShenYin-FranchettaK3}, a key step is an argument using projective bundles, which is further generalized in  \cite{FLV-Franchetta1, FLV-Franchetta2} into a stratified version. Here we provide the following variant allowing base locus, which is the basic tool in our paper. 
 \begin{proposition}[Projective bundle argument: with base locus]
 	\label{prop:ProjectiveBundleArgument} 
 	Let $P$ be a smooth projective variety and  let $E$ be a vector bundle on it. 
	Let $Q\subset P$ be a (possibly singular) closed subvariety.
	Let
	\[ B\ \ \subset\ \bar{B}:= \PP H^0(P, E\otimes \II_{Q}) \]
	denote the Zariski open subset parameterizing smooth dimensionally transversal sections of $E$ vanishing along $Q$, and let $\pi\colon\mathcal{X}\to B$ denote the universal family of zero loci of such sections.
	Assume that $B$ is not empty, and that the sections in $H^0(P, E\otimes \II_{Q})$ globally generate $E$ outside of $Q$. Then 	for any fiber $X$ of $\pi$, we have
	\[ \GDCH^\ast_B (X) = \ima\left(\CH^\ast(P)\to \CH^\ast(X)\right)  + \ima\left( \CH^\ast(Q)\to \CH^\ast(X)\right),\]
	where on the right-hand side, the first morphism is the Gysin restriction map and the second morphism is the push-forward via the natural inclusion. 
\end{proposition} 

\begin{proof} 
	Let  $\bar{\mathcal{X}}\to\bar{B}$ denote the universal family of zero loci of sections. The assumption that  $E$ is globally generated outside of $Q$ by its sections vanishing along $Q$ implies that  the evaluation map $\bar{\mathcal{X}}\to P$ restricts to a projective bundle over the open subset $P\setminus Q$. Reasoning with the projective bundle formula as in \cite[Lemma 1.1]{PavicShenYin-FranchettaK3} or \cite[Proposition 2.6]{FLV-Franchetta2}, this readily gives that
	\[ \ima\bigl(  \CH^\ast(\bar{\mathcal{X}}\setminus (Q\times\bar{B}))\to \CH^\ast(X\setminus Q)\bigr) =
	\ima\bigl( \CH^\ast(P\setminus Q)\to \CH^\ast(X\setminus Q)\bigr)\ .\]       
	By the localization exact sequence for Chow groups, this implies that
	\[ \GDCH^\ast_B (X)\ \subset\   \ima\bigl( \CH^\ast(P)\to \CH^\ast(X)\bigr) +  \ima\bigl( \CH^*(Q)\to \CH^\ast(X)\bigr)\ .\]
	The converse inclusion is obvious.
\end{proof}

The following easy observation abstracts a basic setup that will be repeatedly used in our proof of Theorem \ref{thm:main-2}. In practice, $P$ is some incidence variety in $B\times T$ which dominates $B$.
\begin{lemma}
	\label{lemma:LinkFanoScroll}
	Let $P, B, T$ be varieties  and let $p\colon P\to B$ and $q\colon P\to T$ be morphisms. Let $\pi\colon \mathcal{X}\to B$ be a smooth projective morphism.  For a point $b\in B$ lying in the image of $p$, let $X:=X_b$ be the fiber of $\pi$ over $b$. Then 
	\[\GDCH^*_B(X)\subset \bigcap_{t\in q(p^{-1}(b))}\GDCH^*_{q^{-1}(t)}(X),\]
	where on the right-hand side, $X$ is viewed as a fiber in the base change to $P$ (or rather to $q^{-1}(t)$) of the family $\mathcal{X}/B$.
\end{lemma}
\begin{proof}
For any $(b, t)\in B\times T$ such that $p^{-1}(b)\cap q^{-1}(t)\neq \emptyset$, we have, by restricting $p$, a morphism $q^{-1}(t) \to B$ whose image contains $b$.
Therefore, $\GDCH^*_B(X)\subset \GDCH^*_{q^{-1}(t)}(X)$, where $X=X_b$. One can conclude by letting $t$ run through $q(p^{-1}(b))$.
\end{proof}

\section{Special cubic fourfolds and associated K3 surfaces}
\label{sect:SpecialCubic}

Let $X$ be a cubic fourfold, that is, a smooth hypersurface of degree 3 in $\PP^5$. 
Its middle cohomology group $\HH^4(X, \ZZ)$ equipped with the intersection pairing is naturally a unimodular lattice abstractly isometric to $\mathbf{I}_{21,2}$ and, up to a Tate twist, it also carries a weight-2 Hodge structure of K3 type with Hodge numbers $(1, 21, 1)$.
Denote by $h:=c_1(\mathcal{O}_X(1))\in  \HH^2(X, \ZZ)$ the hyperplane section class. 
The $h$-primitive cohomology group
\[\HH^4(X, \ZZ)_0=\{h^2\}^{\perp}\]
is a Hodge structure of K3 type with Hodge numbers $(1, 20, 1)$, and as a lattice is isometric to the following \textit{cubic lattice}:
\[\Gamma:= E_8^{\oplus 2}\oplus U^{\oplus 2}\oplus A_2,\]
where $E_8$ is the unique positive definite unimodular even lattice of rank 8, $U$ is the hyperbolic plane, and $A_2$ is the lattice with intersection form $\left( \begin{smallmatrix}
2 &-1\\-1 &2
\end{smallmatrix}\right)$. We can fix embeddings without loss of generality (such embeddings are unique up to isometries of $\mathbf{I}_{21, 2}$):
\[h^2\in \mathbf{I}_{21, 2} \quad \text{ and  }\quad \Gamma=\{h^2\}^\perp\subset \mathbf{I}_{21,2}.\]

The moduli space of cubic fourfolds is denoted by $\mathcal{C}:=|\mathcal{O}_{\PP^5}(3)|_{\operatorname{sm}}/\operatorname{PGL}_6$. The local period domain
\begin{equation}
	\label{eqn:OmegaGamma}
	\Omega(\Gamma):=\left\{\omega\in \PP(\Gamma\otimes \C)~\vert~ \omega^2=0, ~ \omega\cdot \overline{\omega}<0\right\}
\end{equation}
is equipped with  a natural action of the group 
\[\wt{O}(\Gamma):=\left\{g\in O(\Gamma)~\vert~ \bar{g}|_{A_\Gamma}=\ide_{A_\Gamma}\right\}=\left\{\tilde{g}\in O(\mathbf{I}_{21,2})~\vert~ \tilde{g}(h^2)=h^2\right\},\]
where $A_\Gamma=\Gamma^\vee/\Gamma$ is the discriminant group of $\Gamma$.
The corresponding quotient is called the \textit{global period domain}
\[\mathcal{D}:=\Omega(\Gamma)/\wt{O}(\Gamma),\]
which is a normal and quasi-projective variety by \cite{BailyBorel}. Sending a cubic fourfold $X$ to its period $H^{3,1}(X)$, we get the period map
$\mathcal{C}\to \mathcal{D}$,
which is shown to be an open immersion by Voisin \cite{VoisinCubicTorelli}.

\subsection{Special cubics}
\label{subsec:SpecialCubics}
Denote the subgroup of integral Hodge classes by
\[\HH^{2,2}(X, \ZZ):=\HH^{2,2}(X)\cap \HH^4(X, \ZZ),\]
which is also the subgroup of algebraic classes, thanks to the integral Hodge conjecture proved by Voisin \cite{Voisin07}. 

For a very general cubic fourfold, $\HH^{2,2}(X, \ZZ)=\ZZ h^2$. Following \cite{Hassett-SpeicialCubic}, a cubic fourfold $X$ is called \textit{special}, if 
$\HH^{2,2}(X, \ZZ)$ is of rank at least two. More precisely, a \textit{marked} cubic fourfold is a (special) cubic fourfold together with a primitive embedding of lattices 
$K\hookrightarrow \HH^{2,2}(X, \ZZ)$ from a rank-2 lattice $K$  such that the image contains $h^2$. A \textit{labelled} cubic fourfold is a cubic fourfold together with a primitive rank-2 sublattice $K\subset \HH^{2,2}(X, \ZZ)$ containing $h^2$. Such an embedding (resp.~a sublattice) is called a \textit{marking} (resp.~a \textit{labelling}), and the determinant of the intersection matrix of $K$ is called the \textit{discriminant} of the (marked or labelled) special cubic fourfold. It turns out (\cite[Proposition 3.2.4]{Hassett-SpeicialCubic}) that the lattice $K$, as well as its embedding into $\mathbf{I}_{21,2}$, is determined by $d$, up to isometries of $\mathbf{I}_{21,2}$ preserving the class $h^2$. Hence it is conventional to denote $K$ by $K_d$, and we can fix without loss of generality the embeddings
$$h^2\in K_d\subset  \mathbf{I}_{21,2},$$
$$\Gamma_d:=K_d^\perp\subset\Gamma.$$

By Hassett \cite[Theorem 1.0.1]{Hassett-SpeicialCubic}),
for a positive integer $d$, there exists a special cubic fourfold of discriminant $d$  if and only if 
\[(*) \quad d>6 \quad \text{ and } \quad d\equiv 0 \text{ or } 2 \pmod 6;\] 
moreover, for such an integer $d$, the locus of special cubic fourfolds of discriminant $d$ is an irreducible divisor in the moduli space $\mathcal{C}$, denoted by $\mathcal{C}_d$. 

The period domains of labelled and marked  cubic fourfolds of discriminant $d$ are 
\begin{align}
	\mathcal{D}^{\lab}_d&=\Omega(\Gamma_d)/\wt{O}(\Gamma, K_d),\\
	\mathcal{D}^{\mar}_d&=\Omega(\Gamma_d)/\wt{O}(\Gamma_d),
	\label{eqn:DefDmard}
\end{align}
where $\Omega(\Gamma_d)$ is defined similarly as in \eqref{eqn:OmegaGamma},  and $\wt{O}(\Gamma, K_d)$ (resp.~$\wt{O}(\Gamma_d)$) is the subgroup of elements of $\wt{O}(\Gamma)$ that preserves (resp.~acts trivially on) the sublattice $K_d$.

Define the moduli spaces of marked and labelled cubic fourfolds of discriminant $d$ respectively as 
$	\mathcal{C}_d^{\mar}:=\mathcal{C}\times_\mathcal{D}\mathcal{D}^{\mar}_d$ and $\mathcal{C}_d^{\lab}:=\mathcal{C}\times_\mathcal{D}\mathcal{D}^{\lab}_d$,
which are normal quasi-projective varieties. There are natural morphisms 
\begin{equation*}
	\xymatrix{
		\mathcal{C}_d^{\mar} \ar@{->>}[r] \ar[d]&\mathcal{C}_d^{\lab}\ar@{->>}[r] \ar[d]& \mathcal{C}_d\ar@{^{(}->}[r]\ar[d]&  \mathcal{C} \ar[d]^{\mathcal{P}}\\
				\mathcal{D}_d^{\mar} \ar@{->>}[r]&\mathcal{D}_d^{\lab}\ar@{->>}[r] & \mathcal{D}_d \ar@{^{(}->}[r]& \mathcal{D}
	}
\end{equation*}
where the  vertical period maps are open immersions, the middle horizonal morphisms are normalizations, while the left horizontal ones are isomorphisms if $d\equiv 2\pmod 6$ and are finite of degree 2 if $d\equiv 0 \pmod 6$ (see \cite[Proposition 5.2.1]{Hassett-SpeicialCubic}).

\subsection{Associated K3} 
\label{subsec:AssociatedK3}
Given a marked special cubic fourfold $(X, K_d\hookrightarrow \HH^{2,2}(X, \ZZ))$ of discriminant $d$, we say a polarized K3 surface $(S, H)$ is 
\textit{Hodge-theoretically associated} to $X$ if there exists a Hodge isometry:
\[\HH^4(X, \ZZ)\supset K_d^\perp\xrightarrow{\simeq} H^\perp(-1)\subset \HH^2(S, \ZZ)(-1),\]
where $(-1)$ is the Tate twist and changes the sign of the quadratic form. Note that by comparing the discriminant, we have $\deg(H^2)=d$, i.e. the K3 surface is of genus $g=\frac{d}{2}+1$.

Hassett \cite[Theorem 5.1.3]{Hassett-SpeicialCubic} showed that a special cubic fourfold of discriminant $d$ has a Hodge-theoretically associated K3 surface if and only if 
\[(**)\quad d \text{ satisfies }(*) \text{ and } d/2 \text{ is not divisible by 9 or any prime number } p \equiv -1\pmod 3.\]
Such $d$'s are called \textit{admissible}, and the first few values are 14, 26, 38, 42, 62, 74, 78, etc.

On the other hand, following \cite{Kuz10}, let
\[\Ku(X):=\left\{E\in \Db(X)~\vert~ \operatorname{RHom}(\mathcal{O}_X(i), E)=0 \text{ for } i=0,1,2\right\}\]
be the \textit{Kuznetsov component} of (the bounded derived category of coherent sheaves of) $X$, which is a 2-Calabi--Yau category.  One says that an algebraic K3 surface $S$ is 
\textit{homologically associated} to $X$ if there is an equivalence of triangulated categories:
	\[\Ku(X)\simeq \Db(S).\]

Both notions of Hodge-theoretically and homologically associated K3 surfaces are very much motivated by the rationality problem of cubic fourfolds, a topic that we do not treat in this paper. However, what is important to us is the following relation between these two notions.

\begin{theorem}[Addington--Thomas {\cite[Theorem 1.1]{AddingtonThomas}}, {\cite[Corollary 1.7]{BLMNPS}}]
	\label{thm:AddingtonThomas}
	Let $d$ be an integer satisfying $(*)$. Let $X\in \mathcal{C}_d$, a special cubic fourfold of discriminant $d$.
	The following conditions are equivalent:
	\begin{enumerate}[(i)]
		\item The integer $d$ is admissible, i.e. it satisfies the condition $(**)$.
		\item $X$ has a homologically associated K3 surface: 
		$\Ku(X)\simeq \Db(S)$ for some projective K3 surface $S$.
	\end{enumerate}
\end{theorem}

The arguments in \cite{AddingtonThomas} and \cite{BLMNPS} actually show that assuming $(i)$, there is a polarized K3 surface of degree $d$ homologically associated to $X$.

Essentially by taking the characteristic classes of the Fourier--Mukai kernel in $(ii)$, 
B\"ulles \cite{Buelles} established the following relation between the motive of a special cubic fourfold and the motive of its associated K3 surface. 
This can be seen as a motivic lifting of the result of Addington and Thomas {\cite[Theorem 1.2]{AddingtonThomas}}. See also \cite[Theorem 3]{FV2} for a different proof resulting in a stronger version taking into account the quadratic space structure.

\begin{theorem}[B\"ulles {\cite[Theorem 0.4]{Buelles}}]
	\label{thm:Buelles}
	Given a special cubic fourfold $X\in \mathcal{C}_d$ with $d$ satisfying $(**)$, there exist a polarized K3 surface $(S, H)$ of degree $d$ and an isomorphism in the category of rational Chow motives $\CHM$:
	\begin{equation}
		\h(X)\simeq \h(S)(-1)\oplus \1\oplus \1(-2)\oplus \1(-4).
	\end{equation}
	In particular, there is an algebraic cycle $Z\in \CH^3(X\times S)$ which  induces an isomorphism of rational Chow groups:
	\begin{equation}
		\label{eqn:CubicK3Chow}
		\CH_1(X)_{\hom}\simeq \CH_0(S)_{\hom}.
	\end{equation}
\end{theorem}
\begin{remark}
	\label{rmk:ChoiciOfCycle}
The original proof in \cite{Buelles} shows that the cycle $Z$ can be chosen to be the codimension-3 component of the Mukai vector of the Fourier--Mukai kernel $\mathcal{E}$ inducing the equivalence between $\Ku(X)$ and $\Db(S)$:
\[Z=v_3(\mathcal{E}).\]
The proof in \cite{FV2} actually shows that $Z':=v_3(\mathcal{E}^R)$ gives the inverse of the isomorphism \eqref{eqn:CubicK3Chow}, where $\mathcal{E}^R:= \mathcal{E}^\vee\otimes p_X^*\omega_X[4]$ is the Fourier--Mukai kernel of the right adjoint. 
\end{remark}

\subsection{Linking two Franchetta properties}
\label{subsect:LinkCubicK3}
The main purpose of this section is the following result, which for an admissible $d$, transforms the generalized Franchetta conjecture \ref{conj:GFC} for degree $d$ K3 surfaces into the Franchetta property for \textit{marked} special cubic fourfolds of discriminant $d$.  Denote by $\mathcal{C}_{d,1}^{\mar}\to \mathcal{C}_{d}^{\mar}$ the universal family of cubic fourfolds and by $\mathcal{F}_{g,1}\to \mathcal{F}_g$ the universal family of K3 surfaces.
\begin{theorem}
	\label{thm:GFCK3andCubic}
	Let $d$ be an integer satisfying the condition $(**)$. Let $g=\frac{d}{2}+1$.
	The Franchetta property for codimension-2 cycles for $\mathcal{F}_{g,1}\to \mathcal{F}_g$ is equivalent to the Franchetta property for codimension 3-cycles for  $\mathcal{C}_{d,1}^{\mar}\to \mathcal{C}_{d}^{\mar}$.
\end{theorem}

\begin{remark}
	Theorem \ref{thm:main-3} is a consequence of Theorem \ref{thm:GFCK3andCubic}, since when $d\equiv 2 \pmod 6$, $\mathcal{C}^{\mar}_d\to \mathcal{C}_d$ is the normalization map, hence does not affect the Franchetta property. On the other hand, if $d\equiv 0 \pmod 6$, then $\mathcal{C}^{\mar}_d\to \mathcal{C}_d$ is of degree 2. Hence by Remark \ref{rmk:Subfamily}, the Franchetta property for $\mathcal{F}_g$ implies the Franchetta property for $\mathcal{C}_d$.
\end{remark}

To prove Theorem \ref{thm:GFCK3andCubic}, it is crucial to adapt Addington--Thomas' Theorem \ref{thm:AddingtonThomas} and B\"ulles' Theorem \ref{thm:Buelles} into their family version. Let 
$\Lambda:=E_8(-1)^{\oplus 2}\oplus U^{\oplus 3}$ be the \textit{K3 lattice}. Let 
$\Lambda_d:=E_8(-1)^{\oplus 2}\oplus U^{\oplus 2}\oplus \ZZ(-d)$ be the abstract lattice underlying the second primitive cohomology of a polarized K3 surface of degree $d$.
Hassett's condition $(**)$ mentioned above is equivalent to the existence of an isometry, up to a sign, between the lattices $\Gamma_d$ and $\Lambda_d$.

For an integer $d$ satisfying $(**)$, upon fixing an isometry $\epsilon: \Gamma_d\xrightarrow{\simeq} \Lambda_d(-1)$, we have an induced isomorphism between the period domain of marked special cubic fourfolds of discriminant $d$ and the period domain of polarized K3 surfaces of degree $d$:
\[\mathcal{D}^{\mar}_d=\Omega(\Gamma_d)/\wt{O}(\Gamma_d)\xrightarrow{\simeq} \mathcal{N}_{d}=\Omega(\Lambda_d)/\wt{O}(\Lambda_d),\]
which gives rise to a birational isomorphism (depending on the choice of $\epsilon$) between the moduli space of marked cubic fourfolds of discriminant $d$ and the moduli space of polarized K3 surfaces of genus $g:=\frac{d}{2}+1$:
\begin{equation}
\label{eqn:phi}
\xymatrix{\phi\colon	\mathcal{C}^{\mar}_d \ar@{-->}[r]^{\simeq}& \mathcal{F}_{g}.
}
\end{equation}
The rational map $\phi$ sends a marked cubic fourfold to its Hodge-theoretically associated polarized K3 surface.

Let $\mathcal{F}^\circ_g$ be a Zariski open subset of $\mathcal{F}_g$ where $\phi$ is an isomorphism. 
The restrictions  over $\mathcal{F}_{g}^\circ$  of  the universal families $\mathcal{C}_{d,1}^{\mar}$ and $\mathcal{F}_{g,1}$ are denoted by $\mathcal{X}\to \mathcal{F}_{g}^\circ$ and $\mathcal{S}\to \mathcal{F}_{g}^\circ$ respectively.

For a cubic fourfold $X$, Addington--Thomas \cite[Definition 2.2]{AddingtonThomas} equipped the topological K-theory of the Kuznetsov component $\Ku(X)$ with a lattice structure via the Euler pairing, abstractly isometric to the Mukai lattice $\wt\Lambda:=E_8^{\oplus 2}\oplus U^{\oplus 4}$, and a natural weight-2 Hodge structure of K3 type via the Mukai-vector map:
\[v\colon \operatorname{K}_{\operatorname{top}}(\Ku(X)) \hookrightarrow \HH^*(X, \QQ).\] 
The resulting \textit{Mukai lattice} of $\Ku(X)$, denoted by $\wt{\HH}(\Ku(X), \ZZ)$, 
always contains the $A_2(-1)$-lattice $\langle\lambda_1, \lambda_2\rangle$, where $\lambda_i$ is the class of  $p(\mathcal{O}_{\operatorname{line}}(i))$, and $p\colon \Db(X)\to \Ku(X)$ is the left adjoint of the inclusion functor $\Ku(X)\hookrightarrow \Db(X)$. Identifying $\wt{\HH}(\Ku(X), \ZZ)$ with its image via $v$, the Mukai vectors $\lambda_i$ are given explicitly as follows, denoted by the same notation:
\begin{align*}
\lambda_1&=3 +\frac{5}{4}h-\frac{7}{32}h^2-\frac{77}{384}h^3+\frac{41}{2048}h^4\,;\\
\lambda_2&=-3-\frac{1}{4}h+\frac{15}{32}h^2+\frac{1}{384}h^3-\frac{153}{2048} h^4.
\end{align*}

Now for the family $\mathcal{X}\to \mathcal{F}^\circ_g$, we have the local system of Mukai lattices over $\mathcal{F}_{g}^\circ$.
\[\mathbb{H}:=\left\{\wt{\HH}(\Ku(X_t), \ZZ)\right\}_{t\in \mathcal{F}_{g}^\circ}\]
\begin{lemma}
	\label{lemma:IsotropicMukaiVector}
	There exist  sections \footnote{By definition, a section of a local system is flat and global (i.e.~monodromy invariant).} $\v, \v', \w$ of the local system $\mathbb{H}$ such that they are fiberwise algebraic and satisfy $\v^2=0$, $\v\cdot \v'=1$, $\v\cdot \w=0$, and $\w^2=-d$.
\end{lemma}
\begin{proof}
	This is essentially \cite[Theorem 3.1]{AddingtonThomas}. Indeed, by the definition of $\mathcal{D}_d^{\mar}$ in \eqref{eqn:DefDmard}, we see that the monodromy invariant subspace of $\HH^*(X, \QQ)$ contains $\langle 1, h, h^2, h^3, h^4 \rangle+K_d$, whose inverse image by the Mukai-vector map $v$, denoted by $L_d$, is the saturation of the lattice $\langle\lambda_1, \lambda_2\rangle\oplus \ZZ v_d$, where $v_d$ is the generator of the orthogonal complement of $h^2$ in $K_d$:
	\[K_d=\overline{\ZZ h^2\oplus \ZZ v_d}\,,\quad L_d=\overline{\langle\lambda_1, \lambda_2\rangle\oplus \ZZ v_d}.\]
	 All classes in $L_d$ are fiberwise Hodge, hence algebraic. By construction, $L_d$ is a rank 3 primitive sublattice in $\wt{\Lambda}$ of discriminant $d$ such that 
	\[\Gamma\supset K_d^\perp=:\Gamma_d=L_d^\perp\subset \wt{\Lambda}\]
    By \cite[Theorem 3.1, $(1)\Rightarrow (2)$]{AddingtonThomas} , or more directly, by \cite[Lemma 1.10, Remark 1.11]{Huy2019Lecture}, there is an isomorphism
   \[L_d\xrightarrow{\simeq} U\oplus \ZZ(-d).\]
   One can then take $\v$, $\v'$ to be the standard basis of $U$ and $\w$ to be the generator of $\ZZ(-d)$.
\end{proof}

\begin{example}
Let us give the explicit formulas of the vectors in the two cases, namely $g=14$ and $22$.
\begin{itemize}
	\item When $g=14$, or equivalently $d=26$, the monodromy invariant part of the local system  $\mathbb{H}$ contains the lattice generated by $\lambda_1, \lambda_2$ and an extra class $\tau$, with the intersection form

\begin{equation}
		\begin{array}{c|ccc}
	&\lambda_1&\lambda_2& \tau\\
	\hline
	\lambda_1&-2&1&0\\
	\lambda_2&1&-2&1\\
	\tau&0&1&8\\
	\end{array}
\end{equation}
Then we take $\v=\lambda_1+3\lambda_2+\tau$, $\v'=\lambda_1$ and $\w=11\lambda_1+22\lambda_2+7\tau$.

\item When $g=22$, or equivalently $d=42$, the monodromy invariant part of the local system  $\mathbb{H}$ contains the lattice generated by $\lambda_1, \lambda_2$ and an extra class $\tau=v_{42}$, with the intersection form

\begin{equation}
\begin{array}{c|ccc}
&\lambda_1&\lambda_2& \tau\\
\hline
\lambda_1&-2&1&0\\
\lambda_2&1&-2&0\\
\tau&0&0&14\\
\end{array}
\end{equation}
Then we take $\v=\lambda_1+3\lambda_2+\tau$, $\v'=\lambda_1$ and $\w=14\lambda_1+28\lambda_2+9\tau$.
\end{itemize}
\end{example}

Now we can extend Addington--Thomas' result Theorem \ref{thm:AddingtonThomas} into the following family version.
\begin{proposition}
	\label{prop:AddingtonThomasInFamily}
	Let $d$ be an integer satisfying $(**)$ and $g=\frac{d}{2}+1$.
	Let $\mathcal{X}$ and $\mathcal{S}$ be the family of cubic fourfolds and K3 surfaces over $\mathcal{F}_g^\circ$ as above. Up to replacing $\mathcal{F}_g^\circ$ by a non-empty Zariski open subset, there exists a relative Fourier--Mukai kernel
	$\mathcal{E}\in \Db(\mathcal{X}\times_{\mathcal{F}_g^\circ}\mathcal{S})$ such that for any $t\in \mathcal{F}_g^\circ$, the Fourier--Mukai transform with kernel $\mathcal{E}_t\in \Db(X_t\times S_t)$ induces an equivalence $\Ku(X_t)\xrightarrow{\simeq}\Db(S_t)$.
\end{proposition}
\begin{proof}
	A distinguished connected component of the (numerical) stability manifold of cubic fourfolds is constructed in \cite{BLMS}. Let $\v, \v', \w$ be as in Lemma \ref{lemma:IsotropicMukaiVector}. By \cite[Theorem 29.4]{BLMNPS}, for a $\v$-generic stability condition $\underline{\sigma}$ on $\mathcal{X}$ over $\mathcal{F}_g^\circ$,
	there is a relative moduli space $\mathcal{M}_{\underline{\sigma}}(\v)$  of Bridgeland stable objects in $\Ku(\mathcal{X}/\mathcal{F}_g^\circ)$ with Mukai vector $\v$, which is (up to shrinking $\mathcal{F}_g^\circ$) a relative projective K3 surface  over $\mathcal{F}_g^\circ$. By the existence of the vector $\v'$ with $\v\cdot \v'=1$, this moduli space $\mathcal{M}_{\underline{\sigma}}(\v)$ is fine. The existence of the vector $\w$ with $\w^2=d$ implies that $\mathcal{M}_{\underline{\sigma}}(\v)$ admits a relative polarization over $\mathcal{F}_g^\circ$ of degree $d$. We can therefore identify $\mathcal{S}$ with $\mathcal{M}_{\underline{\sigma}}(\v)$. Let $\mathcal{E}$ be the universal sheaf. Then the corresponding Fourier--Mukai transform is an equivalence by \cite[Lemma 33.2]{BLMNPS}.
\end{proof}

We deduce the following family version of B\"ulles' result.
\begin{corollary}
	\label{cor:MotiveCubicK3}
	Let the notation be as before. Up to shrinking $\mathcal{F}_g^\circ$, there exist cycles 
	$Z \in \CH^3(\mathcal{X}\times_{\mathcal{F}_g^\circ}\mathcal{S})$ and $Z^\prime\in \CH^3(\mathcal{S}\times_{\mathcal{F}_g^\circ}\mathcal{X})$,
	with the property that for any $t\in \mathcal{F}_g^\circ$, the cycles $Z_t, Z'_t\in \CH^3(X_t\times S_t)$ induce mutually inverse  isomorphisms \[\CH_1(X_t)_{\hom}\simeq \CH_0(S_t)_{\hom}.\]
\end{corollary}
\begin{proof}
Let $\mathcal{E}$ be as in Proposition \ref{prop:AddingtonThomasInFamily} and let $\mathcal{E}^R$ be the relative Fourier--Mukai kernel of the right adjoint. Then 
 Theorem \ref{thm:Buelles} and Remark \ref{rmk:ChoiciOfCycle} show that  $Z:=v_3(\mathcal{E})$ and $Z':=v_3(\mathcal{E}^R)$ induce fiberwise inverse isomorphisms between $\CH_1(X_t)_{\hom}$ and $\CH_0(S_t)_{\hom}$.
\end{proof}

\begin{remark}
	By applying the argument (Manin's identity principle) as in the proof of \cite[Theorem 0.4]{Buelles} to the relative Fourier--Mukai kernel $\mathcal{E}$ as well as its right adjoint, we can also show that there is an isomorphism between $\h(\mathcal{X})$ and $\h(\mathcal{S})(-1)\oplus \1\oplus \1(-2)\oplus \1(-4)$, as relative Chow motives  over $\mathcal{F}_g^\circ$.
\end{remark}

\begin{proof}[Proof of Theorem \ref{thm:GFCK3andCubic} (hence Theorem \ref{thm:main-3})]
	Given a point $t\in \mathcal{F}_g^\circ$,
consider the following commutative diagram where the vertical arrows are Gysin restriction maps:
	\begin{equation}
		\xymatrix{
	\CH^3(\mathcal{X}) \ar[d]^{r_1}\ar[r]^{Z_*}& \CH^2(\mathcal{S})\ar[d]^{r_2}\\
	\CH^3(X_t)\ar[r]^{Z_{t,*}} & \CH^2(S_t).
	}
	\end{equation}
	Assume first the Franchetta property  for $\mathcal{F}_{g,1}/\mathcal{F}_g$. For any $\alpha\in \ima(r_1)\cap \CH^3(X_t)_{\hom}$, the above diagram shows that $Z_{t, *}(\alpha)\in \CH^2(S_t)_{\hom}\cap \ima(r_2)$, hence is zero by assumption. By Corollary \ref{cor:MotiveCubicK3}, $Z_{t,*}$ is an isomorphism, thus $\alpha=0\in \CH^3(X_t)$, i.e. the Franchetta property is satisfied for $\mathcal{C}^{\mar}_{d,1}/\mathcal{C}^{\mar}_d$.
	
	Similarly, by using $Z'$ in Corollary \ref{cor:MotiveCubicK3}, one can show that the Franchetta property for $\mathcal{C}^{\mar}_{d,1}/\mathcal{C}^{\mar}_d$ implies that for $\mathcal{F}_{g,1}/\mathcal{F}_g$.
\end{proof}

\begin{proof}[Proof of Theorem \ref{thm:main-2} for $d=14, 38$]
	Since 14 and 38 are both $\equiv  2 \pmod 6$, Theorem \ref{thm:main-3} applies. Therefore, the Franchetta property for the universal family over $\mathcal{C}_{14}$ and $\mathcal{C}_{38}$ are equivalent to Conjecture \ref{conj:GFC} for $g=8$ and $20$ respectively, which are proved in \cite{PavicShenYin-FranchettaK3}.
\end{proof}

\section{Franchetta for $\mathcal{C}_8$ and $\mathcal{F}_{2}$}
\label{sec:d=8}
In this section, we first show Theorem \ref{thm:main-2} for $d=8$, and then deduce from it a new proof of Conjecture \ref{conj:GFC} for $g=2$. The key is the geometric characterization of special cubic fourfolds of discriminant 8: those are exactly the ones containing  a plane \cite[Section 4.1.1]{Hassett-SpeicialCubic}.

Consider the following varieties.
\begin{align*}
B&:=\left\{X\subset \PP^5 ~|~ X \text{ is a cubic fourfold containing a plane}\right\}.\\
P&:=\left\{(R,  X)~|~ X \text{ is a cubic fourfold}, R \text{ is a plane contained in } X \right\}.
\end{align*}
We have natural morphisms $p\colon P\to B$ and $q\colon P\to \Gr(\PP^2, \PP^5)$ sending a couple $(R, X)$ to $X$ and $R$ respectively.  
By construction, $p, q$ are surjective, and the fiber of $q$ over a point $[R]\in \Gr(\PP^2, \PP^5)$ parametrizes all the cubic fourfolds containing the plane $R$, which is a Zariski open subset of 
$\PP H^0(\PP^5, \mathcal{I}_R\otimes \mathcal{O}(3))\simeq \PP^{45}$.

\begin{proof}[Proof of Theorem \ref{thm:main-2} for $d=8$]
	As there is a dominant morphism $B\to \mathcal{C}_8$, it suffices to show the Franchetta property for the universal family of cubic fourfolds $\mathcal{X}\to B$. 
	For any $b\in B$, denote by $X_b$ the corresponding special cubic fourfold of discriminant 8, and let $R$ be any plane contained in $X_b$ (for generic $b$, there is only one plane). It is obvious (or one uses Lemma \ref{lemma:LinkFanoScroll}) that 
	\[\GDCH_B^3(X_b)\subset \GDCH_{B_R}^3(X_b),\]
where $B_R\subset B$ is the subfamily of cubic fourfolds containing the plane $R$. 

By Proposition \ref{prop:ProjectiveBundleArgument}, 
\[\GDCH_{B_R}^3(X_b)=\ima\left(\CH^3(\PP^5)\to \CH^3(X_b)\right)+\ima\left(\CH^1(R)\to \CH^3(X_b)\right)=\QQ h^3+\QQ l,\]
where $h$ is the hyperplane section class and $l$ is the class of a line in $R$. However $l$ and $h^3$ are proportional. Indeed, denoting by $i\colon R\hookrightarrow X$ and $\iota\colon X\hookrightarrow \PP^5$ the natural closed immersions, we have
\begin{equation}
\label{eqn:Relationh3andl}
	h^3=\iota^*(\iota_*(R))=R\cdot c_1(\mathscr{N}_{X/\PP^5})=R\cdot 3h=3i_*i^*(h)=3l.
\end{equation}
Therefore, $\GDCH_B^3(X_b)=\QQ h^3$.
\end{proof}

As an application, we provide a proof of the generalized Franchetta conjecture \ref{conj:GFC} for $g=2$, which is different from the one in \cite{PavicShenYin-FranchettaK3} using Mukai models.

\begin{proof}[Proof of Conjecture \ref{conj:GFC} for $g=2$]
A generic cubic fourfold $X$ in $\mathcal{C}_8$ contains only one plane, denoted by $R$. Projecting from $R$ endows the blow-up $X':=\operatorname{Bl}_{R}X$ with a quadric fibration structure $\pi\colon X'\to \mathbb{P}^2$, where the base $\mathbb{P}^2$ parameterizes all $\mathbb{P}^3$'s containing $R$, and the fibers of $\pi$ are exactly the quadric surfaces that are residual intersections (to $R$) of the corresponding $\mathbb{P}^3$ with $X$. The Stein factorization of the relative Hilbert scheme of lines of $\pi$ is as follows
	\[\operatorname{Hilb}^{\operatorname{line}}(X'/\mathbb{P}^2)\to S \to \mathbb{P}^2,\]
	where the first map is a $\mathbb{P}^1$-fibration and the second map is a double cover. The surface $S$ is the associated (twisted) K3 surface (see \cite[\S1]{VoisinCubicTorelli}, \cite{Kuz10}). We identify $\CH_0(\operatorname{Hilb}^{\operatorname{line}}(X'/\mathbb{P}^2))$ and $\CH_0(S)$. Note that there is a natural map $i\colon \operatorname{Hilb}^{\operatorname{line}}(X'/\mathbb{P}^2) \to F(X)$, providing a uniruled divisor in the Fano variety of lines. By \cite[Example 1.5]{Shen-Yin-cubic}, the following composition is an isomorphism:
	\[\CH_0(S)\simeq \CH_0(\operatorname{Hilb}^{\operatorname{line}}(X'/\mathbb{P}^2))\to \CH_0(F(X))\to \CH_1(X),\]
	where the first map is induced by $i$ and the second map is induced by the incidence variety $\{(l, x)\in F(X)\times X~|~ x\in l\}$.
	
	It is clear from the above construction that the isomorphism between $\CH_0(S)$ and $\CH_1(X)$ can be defined generically over the moduli space $\mathcal{C}_8$, which admits a dominant map to $\mathcal{F}_2$. Therefore the Franchetta property for the universal family of cubic fourfolds over $\mathcal{C}_8$, which is just proved previously, implies the generalized Franchetta conjecture \ref{conj:GFC} for $\mathcal{F}_2$. 
\end{proof}

\begin{remark}[Twisted K3 surfaces]
	Recently, Brakkee \cite{BrakkeeTwistedK3} constructed and studied moduli spaces of twisted polarized K3 surfaces, as well as their relations with special cubic fourfolds.  In particular, the following is shown (\cite[p.~1475]{BrakkeeTwistedK3}): let $g$ be a positive integer such that $d=2g-2$ satisfies $(**)$ and $d\equiv 2 \pmod 6$, then for any $r$ not divisible by 3, there exists a birational isomorphism between the moduli space of special cubic fourfolds $\mathcal{C}_{dr^2}$ and the moduli space $\mathcal{F}_{g}[r]$ of order-$r$ twisted K3 surfaces of genus $g$. Note that forgetting the Brauer class gives rise to a natural surjective map $\mathcal{F}_{g}[r]\to \mathcal{F}_g$. The same argument as in Section \ref{sect:SpecialCubic}, in particular Theorem \ref{thm:GFCK3andCubic}, can be adapted to the twisted case to show that \textit{for $d$ and $g$ as before, the Franchetta property for the universal family over $\mathcal{C}_{dr^2}$ is equivalent to the Franchetta property for the universal family of K3 surfaces over $\mathcal{F}_g[r]$, hence implies the generalized Franchetta conjecture \ref{conj:GFC} for $g$.}
	Our new proof of Conjecture \ref{conj:GFC} for $g=2$ given above is the special case where $g=r=d=2$.
\end{remark}

\section{Franchetta for $\mathcal{C}_{20}$}
\label{sec:d=20}
The $d=20$ case of Theorem \ref{thm:main-2} is already proved in \cite[Lemma 6.3]{FLV-MCK} using the so-called K\"uchle fourfolds of type c7. In this section, we give an alternative proof, which is very similar to the case $d=8$ treated in Section \ref{sec:d=8}. The geometric input is Hassett's result \cite[Section 4.1.4]{Hassett-SpeicialCubic} that special cubic fourfolds of discriminant 20 are characterized generically as the ones containing a Veronese surface, that is, the image of the embedding of  $\PP^2$ into $\PP^5$ via the complete linear system $|\mathcal{O}_{\PP^2}(2)|$. Similarly as in Section \ref{sec:d=8}, set
\begin{align*}
B&:=\left\{X\subset \PP^5 ~|~ X \text{ is a cubic fourfold containing a Veronese surface}\right\},\\
T&:=\left\{R\subset \PP^5 ~|~ R \text{ is a Veronese surface}\right\},\\
P&:=\left\{(R,  X)~|~ X \text{ is a cubic fourfold}, R \text{ is a Veronese surface contained in } X \right\},
\end{align*}
together with natural surjective morphisms $p\colon P\to B$ and $q\colon P\to T$. 

\begin{proof}[Proof of Theorem \ref{thm:main-2} for $d=20$]
Let $\mathcal{X}\to B$ be the universal family. For any $b\in B$, denote the fiber by $X_b$ and take a Veronese surface $R\subset X_b$. Let $B_R\subset B$ be the subvariety parametrizing cubic fourfolds containing $R$, then
	\[\GDCH_B^3(X_b)\subset \GDCH_{B_R}^3(X_b).\]
As $R$ is cut out by quadrics, for any point $x\in \mathbb{P}^5\backslash R$, there exists a cubic fourfold containing $R$ and avoiding $x$. Hence Proposition \ref{prop:ProjectiveBundleArgument} applies and gives that
\[\GDCH_{B_R}^3(X_b)=\ima\left(\CH^3(\PP^5)\to \CH^3(X_b)\right)+\ima\left(\CH^1(R)\to \CH^3(X_b)\right)=\QQ h^3+\QQ l,\]
where $l$ is a line in $R$ (so a conic in $\PP^5$). A similar computation as in \eqref{eqn:Relationh3andl} gives that $3l=2h^3$.
Therefore, $\GDCH_B^3(X_b)=\QQ h^3$.	
\end{proof}

\section{Franchetta for $\mathcal{C}_{26}$ and $\mathcal{F}_{14}$}
\label{sec:g=14}
In this section, we establish Conjecture \ref{conj:GFC} in the case $g=14$. Thanks to Theorem \ref{thm:GFCK3andCubic} (or Theorem \ref{thm:main-3}), it is equivalent to proving the Theorem \ref{thm:main-2} for $d=26$.
%
The key ingredient in our argument is the following geometric characterization of such special cubic fourfolds generically as the ones containing rational normal scrolls of degree 7 with 3 nodes. In the sequel, we often simply call such scrolls \textit{3-nodal and septic}.

\begin{theorem}[Farkas--Verra \cite{FarkasVerra-genus14}]
	\label{thm:FarkasVerraGenus14}
A generic member $X\in \mathcal{C}_{26}$ contains a 2-dimensional family of 3-nodal septic scrolls, parameterized by a non-empty Zariski open subset of the Hodge-theoretically associated degree 14 K3 surface of $X$. Conversely, given a 3-nodal septic scroll $R\subset \PP^5$, a cubic fourfold containing $R$ is special of discriminant 26.
\end{theorem}

Another key ingredient is on the defining equations of these scrolls:

\begin{lemma}[Russo--Staglian\`o \cite{RussoStagliano-C42}]\label{lemma:baselocus} Let $R\subset \PP^5$ be a 3-nodal septic scroll. Then $R$ is cut out by cubic equations.
\end{lemma}

\begin{proof} This has been checked in \cite[Section 7]{RussoStagliano-C42}, cf. item $(\rom2)$ in Table 1 of loc. cit.
\end{proof}

Let us now consider the following parameter spaces:
\begin{align*}
		T&:=\left\{R\subset \PP^5~|~ R \text{ is a 3-nodal septic scroll} \right\}.\\
		B&:=\left\{X\subset \PP^5 ~|~ X \text{ is a cubic fourfold containing a 3-nodal septic scroll}\right\}.\\
		P&:=\left\{R\subset X\subset \PP^5~|~ X \text{ is a cubic fourfold}, R\in T \right\}.
\end{align*}
We emphasize that in the above definitions, we do not quotient out by automorphisms, hence the spaces are some open subsets of certain Hilbert schemes in $\PP^5$.


Then we have natural morphisms in the following diagram. 
\begin{equation}
\label{diag:PBT}
	\xymatrix{
	&P \ar[dr]^{q} \ar[dl]_{p}&\\
B&&T.}
\end{equation} 
By Theorem \ref{thm:FarkasVerraGenus14} (combined with \cite[Proposition 3.4]{FarkasVerra-genus14}), 
we have the following.
\begin{lemma}
	In the above diagram.
	\begin{enumerate}[(i)]
		\item The natural map $B\to \mathcal{C}_{26}$ is dominant.
		\item 	The morphism $p$ is surjective. Its general fibers are Zariski open subsets of K3 surfaces.
		\item  The morphism $q$ is surjective. Its general fibers are Zariski open subsets of  $\PP^{12}$.
	\end{enumerate}
\end{lemma}

Let $\pi\colon \mathcal{X}\to B$ be the  universal family of cubic fourfolds  over $B$. 
\begin{proposition}
	\label{prop:ReduceToRuling}
	For any $b\in B$, let $X_b$ be the fiber of $\pi$ over $b$. Then
	\[\GDCH^3_B(X_b)\subset \bigcap_{t\in q(p^{-1}(b))} \left(\QQ h^3+ \QQ\ell_t\right),\]
	where $\ell_t$ is the class in $\CH^3(X_b)$ of  the ruling of $R_t$, and $R_t$ is the scroll parameterized by $t\in T$.
\end{proposition}
\begin{proof}
	Applying Lemma \ref{lemma:LinkFanoScroll} to the diagram \eqref{diag:PBT}, we have that for any $b\in B$, 
	\[\GDCH^3_B(X_b)\subset \bigcap_{t\in q(p^{-1}(b))} \GDCH^3_{q^{-1}(t)}(X_b).\]
	However, $q^{-1}(t)$ parameterizes all cubic fourfolds containing $R_t$, which is an open subset of $\PP\HH^0(\PP^5, \mathcal{O}(3)\otimes \mathcal{I}_{R_t})\simeq \PP^{12}$.
	Lemma \ref{lemma:baselocus} guarantees that for any point outside of $R_t$, there is a cubic fourfold containing $R_t$ but not this point. Therefore, Proposition \ref{prop:ProjectiveBundleArgument} implies that 
	\begin{equation}
	\label{eqn:GDCHsubfamily}
			\GDCH^3_{q^{-1}(t)}(X_b)= \ima\left(\CH^3(\PP^5)\to \CH^3(X_b)\right)  + \ima\left( \CH^1(R_t)\to \CH^3(X_b)\right).
	\end{equation}
On the right-hand side of \eqref{eqn:GDCHsubfamily}, the first term is obviously $\QQ h^3$. As for the second term, since there is a surjection $\FFF_1\to R_{t}$ (see \cite[Section 3]{FarkasVerra-genus14}), where $\FFF_1=\operatorname{Bl}_{o}\PP^2$ is the first Hirzebruch surface, the group $\CH^1(R_{t})$ is at most 2-dimensional, generated by the restriction
$h\vert_{R_{t}}$ and the class $\ell_t$ of the ruling of the scroll. The class $h|_{R_{t}}$, when pushed-forward to $X_b$, is $h\cdot R_{t}$. 
To conclude, it suffices to show that $h\cdot R_t\in \CH^3(X_b)$ is proportional to $h^3$. To this end,  let $\iota\colon X_b\to\PP^5$ be the natural inclusion. Then we have
\[3h\cdot R_{t}=\iota^\ast \iota_\ast (R_{t}).\]
Since $\iota_*(R_t)=7H^3\in \CH^3(\PP^5)$, where $H$ is the hyperplane class of $\PP^5$, we obtain that $3h\cdot R_t=7h^3$. The proof is complete. 
\end{proof}

Now we are ready to prove the main results, Theorem \ref{thm:main-1}, or equivalently, Theorem \ref{thm:main-2} for $d=26$.
\begin{proof}[Proof of Theorem \ref{thm:main-2} for $d=26$]
	Since there is a dominant morphism $B\to \mathcal{C}_{26}$, it is enough to show the Franchetta property for codimension-3 cycles  for the universal family of special cubic fourfolds $\pi\colon \mathcal{X}\to B$. Thanks to Proposition \ref{prop:ReduceToRuling}, it suffices to show that for a general cubic fourfold $X$ of discriminant $26$, there exists a 3-nodal septic scroll $R\subset X$, such that the class of the ruling $\ell$ of $R$, viewed as an element in $\CH^3(X)$, is proportional to $h^3$.
	
	Let $S$ be the  K3 surface that is Hodge-theoretically associated to $X$. By \cite{FarkasVerra-genus14} (see Theorem \ref{thm:FarkasVerraGenus14}), there is a dense open subset  $S_0\subset S$ parameterizing the 3-nodal septic scrolls contained in $X$. Choose a constant cycle curve $C$ intersecting $S_0$, which is possible because constant cycle curves are Zariski dense in $S$ (see for example \cite[Lemma 2.3]{VoisinK3}).
	For any $t\in C\cap S_0$, let $R_t$ be the corresponding scroll in $X$. Since all rulings of $R_t$ are parameterized by a rational curve $T_t$, we can view $T_t$ as a rational curve in  the Fano variety of lines $F(X)$. Therefore, we have well-defined (i.e. independent of $t\in C$) cycle classes $L :=L_t\in \CH_0(F(X))$ and $\ell:=P_*(L)\in \CH^3(X)$, where  $P\subset F(X)\times X$ is the incidence subvariety (i.e. the universal projective line).
	
	We claim that \textit{the class $L\in \CH_0(F(X))$ has a 2-dimensional rational orbit}.  Indeed, 
	by \cite{Hassett-SpeicialCubic}, for $d=26$, there is an isomorphism $$\varphi\colon S^{[2]} \xrightarrow{\simeq} F(X)$$ between the Hilbert square of $S$ and $F(X)$.
Since $C$ is a constant cycle curve in $S$, we have the following constant cycle surface  in $S^{[2]}$:
\[W:=\{z\in S^{[2]}~|~ \operatorname{supp}(z)=\{t\}, t\in C\},\]
whose image under $\varphi$ gives rise to a constant cycle surface in $F(X)$. To prove the claim, we only need to see that the points of this constant cycle surface $\varphi(W)$ represent the class $L\in \CH_0(F(X))$. To this end, let $\rho\colon S^{[2]}\to S^{(2)}$ be the Hilbert--Chow morphism, then by the construction of \cite{FarkasVerra-genus14}, for any $t\in S$, the septic rational curve $T_t\subset F(X)$ parameterizing the rulings of $R_t$ is exactly $\varphi(\rho^{-1}(t))$, where $t$ is viewed as a point of the diagonal $\Delta_S\subset S^{(2)}$. Hence the class of points on $\varphi(W)$ is $L$.
The claim is proved. In other words, $L\in \mathsf{S}_2\CH_0(F(X))$, where $\mathsf{S}_\bullet$ refers to Voisin's orbit filtration on 0-cycles \cite{Voisin-Coisotropic}. 

However, thanks to Voisin's result  \cite[Proposition 4.5]{Voisin-Coisotropic} (or \cite[Theorem 2.5]{Voisin-Coisotropic}), we know that $\mathsf{S}_2\CH_0(F(X))$ is one-dimensional, generated by $g^4$, where $g$ is the Pl\"ucker polarization class of $F(X)$. Hence $L\in \QQ g^4$ in $\CH_0(F(X))$.
	
	Since the incidence subvariety $P$ in $F(X)\times X$ induces a morphism  
	\[ P_*\colon \CH_0(F(X))\ \to\CH^3(X), \]  
	which sends $g^4$ to $36 h^3$ (see for example \cite[Lemma A.4]{SV}) and $P_*(L)=\ell$ by construction, one can conclude that $\ell\in \QQ h^3$. In other words, the ruling class $\ell_t$ is a proportional $h^3$ for any $t\in C\cap S_0$. The proof is complete.
\end{proof}

\section{Franchetta for $\mathcal{C}_{14}$ and $\mathcal{F}_{8}$}
\label{sect:d=14}

The argument in Section \ref{sec:g=14} can also be applied to give a new proof of Conjecture \ref{conj:GFC} for $g=8$, or equivalently (by Theorem \ref{thm:GFCK3andCubic}), the Franchetta property for the universal family of special cubic fourfolds over $\mathcal{C}_{14}$. 

Recall that a generic cubic fourfold in $\mathcal{C}_{14}$ is a \textit{Pfaffian} cubic (see Beauville--Donagi \cite{BeauvilleDonagi}), namely, a 4-dimensional smooth linear section of the Pfaffian cubic hypersurface 
$$\operatorname{Pf}:=\left\{\phi\in \mathbb{P}(\bigwedge^2 V)~|~ \phi\wedge\phi\wedge \phi =0\right\},$$
where $V$ is a 6-dimensional vector space. The associated K3 surface $S$ is the dual 2-dimensional linear section of $\Gr(2, V^\vee)\subset \mathbb{P}(\bigwedge^2 V^\vee)$. The key ingredient is the following characterization of cubic fourfolds in $\mathcal{C}_{14}$ by smooth rational normal quartic scrolls, simply called \textit{quartic scrolls} in the sequel, in analogy with Theorem \ref{thm:FarkasVerraGenus14}. 

\begin{theorem}[Hassett {\cite[4.1.3]{Hassett-SpeicialCubic}}, {Beauville--Donagi \cite[Section 2]{BeauvilleDonagi}}, Tregub \cite{Tregub}]
	\label{thm:PfaffCubic}
	Let $\mathcal{C}_{14}$ be the moduli space of special cubic fourfolds with discriminant 14.
	\begin{enumerate}[(i)]
		\item 	A generic member $X$ in $\mathcal{C}_{14}$ is Pfaffian and contains a quartic scroll and conversely, a cubic fourfold containing a quartic scroll is in $\mathcal{C}_{14}$.
		\item A Pfaffian cubic fourfold $X$ contains a two-dimensional family of quartic scrolls parameterized by the associated K3 surface $S$. Moreover, there is a natural isomorphism $S^{[2]}\simeq F(X)$.
	\end{enumerate}
\end{theorem}

Another geometric fact we need is the following, see for example \cite[\S 1.4]{Hassett-CIMElecture}.
\begin{lemma}\label{lemma:QuarticScrollEquation}
	A  quartic scroll in $\mathbb{P}^5$ is cut out by quadric equations.
\end{lemma}
\begin{proof}
	In fact, a quartic scroll in $\mathbb{P}^5$ can be defined by the $2\times 2$ minors of the matrix 
	\begin{equation}
		\begin{pmatrix}
			u & v & x & y\\
			v & w & y & z
		\end{pmatrix}
	\end{equation}
where $[u: v: w: x: y: z]$ are the homogeneous oordinates of $\mathbb{P}^5$.
\end{proof}

\begin{proof}[Proof of Theorem \ref{thm:main-2} for $d=14$]
	Consider
	\begin{align*}
		B&:=\left\{X\subset \PP^5 ~|~ X \text{ is a Pfaffian cubic fourfold}\right\};\\
		T&:=\left\{R\subset \PP^5~|~ R \text{ is a quartic scroll} \right\};\\
		P&:=\left\{R\subset X\subset \PP^5~|~ X \text{ is a cubic fourfold}, R\in T \right\},
	\end{align*}
together with natural morphisms $p\colon P\to B$ and $q\colon P\to T$, which are surjective by Theorem \ref{thm:PfaffCubic} $(i)$. Since $B\to \mathcal{C}_{14}$ is dominant, it suffices to show the Franchetta property for the universal family of cubic fourfolds $\mathcal{X}\to B$. 

Similarly to Proposition \ref{prop:ReduceToRuling}, we first show that 
\begin{equation}
	\label{eqn:PfaffianToRuling}
	\GDCH^3_B(X_b)\subset \bigcap_{t\in q(p^{-1}(b))} \left(\QQ h^3+ \QQ\ell_t\right),
\end{equation}
where $\ell_t\in \CH^3(X_b)$ of  the class of a ruling of the scroll $R_t$, for any $t\in T$. Indeed, Lemma \ref{lemma:LinkFanoScroll} yields that 
for any $b\in B$, 
\[\GDCH^3_B(X_b)\subset \bigcap_{t\in q(p^{-1}(b))} \GDCH^3_{q^{-1}(t)}(X_b);\] 
while for any $t\in q(p^{-1}(b))$,  Lemma \ref{lemma:QuarticScrollEquation} allows us to apply Proposition \ref{prop:ProjectiveBundleArgument} to obtain that 
\[	\GDCH^3_{q^{-1}(t)}(X_b)= \ima\left(\CH^3(\PP^5)\to \CH^3(X_b)\right)  + \ima\left( \CH^1(R_t)\to \CH^3(X_b)\right).\]
On the right-hand side, the first term gives $\mathbb{Q}h^3$, and the second term is generated by the push-forward of $h|_{R_t}$ and $\ell_t$, since $R_t$ is a rational ruled surface. A similar computation as in Proposition \ref{prop:ReduceToRuling} yields that the push-forward of
$h|_{R_t}$ is $\frac{4}{3}h^3$ in $\CH_1(X_b)$. The equality \eqref{eqn:PfaffianToRuling} is proved.

It remains to show that for any Pfaffian cubic fourfold $X$, there exists a quartic scroll $R\subset X$, such that the class of the ruling $\ell$ of $R$, viewed in $\CH^3(X)$, is proportional to $h^3$. The argument is as in the proof in Section \ref{sec:g=14} for the $d=26$ case of Theorem \ref{thm:main-2}. Let $S$ be the associated K3 surface. Choose a (sufficiently generic) constant cycle curve $C$ in $S$, then the rulings of the scrolls parametrized by $t\in C$ (see Theorem \ref{thm:PfaffCubic} $(ii)$) all represent the same classes $L\in \CH_0(F(X))$ and $\ell\in \CH^3(X)$.

The constant cycle curve $C$ gives rise to a constant cycle surface in $S^{[2]}$:
\[  W:= \{z\in S^{[2]}~|~ \operatorname{supp}(z)=\{t\}, t\in C\}.\]
Using the isomorphism $\varphi\colon S^{[2]}\simeq F(X)$ (Theorem \ref{thm:PfaffCubic} $(ii)$), we obtain a constant cycle surface in $F(X)$. One can check from the explicit construction of the isomorphism $ \varphi$
 given in \cite[Proposition 5]{BeauvilleDonagi} that the rational curve in $F(X)$ corresponding to the family of rulings of the scroll $R_t$ parameterized by $t\in S$, is exactly the image under $\varphi$ of $\rho^{-1}(t)\simeq \mathbb{P}^1$, where $t$ is viewed as a point in $\Delta_S\subset S^{(2)}$ and $\rho\colon S^{[2]}\to S^{(2)}$ is the Hilbert--Chow morphism.
It follows that for any point $w\in W$, the 0-cycle class $w\in \CH_0(S^{[2]})$ (which does not depend on $w$ as $W$ is a constant cycle surface) maps via $\varphi$ to $L\in \CH_0(F(X))$. Therefore, the class $L$ has 2-dimensional rational orbit, hence must be a multiple of $g^4$ by Voisin \cite[Proposition 4.5]{Voisin-Coisotropic}. By \cite[Lemma A.4]{SV}, we conclude that $\ell=P_*(L)$ is a multiple of $h^3$, as desired.	
\end{proof}

\begin{corollary}
Conjecture \ref{conj:GFC} holds for $g=8$.
\end{corollary}
\begin{proof}
	By Theorem \ref{thm:main-3}, it follows from the Franchetta property for special cubic fourfolds in $\mathcal{C}_{14}$, which has just been proved. 
	
	(We remark that instead of appealing to the general result Theorem \ref{thm:main-3}, the second author has established in \cite[Corollary 4.4]{Laterveer-PfaffianGrassmannian} directly the link between the $\CH_1$ of a Pfaffian cubic fourfold and the $\CH_0$ of the associated K3 surface, which is generically defined. This avoids the use of techniques from derived categories.)
\end{proof}

\bibliographystyle{amsalpha}
\bibliography{bib}

\end{document}